\documentclass[12pt]{amsart}

\usepackage{fullpage}

\usepackage{array}
\usepackage{amsmath, amsthm, amssymb}
\usepackage{thmtools,thm-restate}
\usepackage{mathtools}
\mathtoolsset{centercolon}
\usepackage{enumitem}
\usepackage[style=alphabetic,sorting=nty,firstinits=true,maxnames=10,backend=bibtex8]{biblatex}
\usepackage{bm}
\usepackage{upgreek}
\usepackage[utf8]{inputenc}
\DeclareUnicodeCharacter{012D}{\u\i}
\usepackage[hidelinks]{hyperref}
\hypersetup{
  pdftitle={Thompson's theorem for compact operators and diagonals of unitary operators},
  pdfauthor={John Jasper, Jireh Loreaux, Gary Weiss}
}

\newcommand{\vabs}[1]{\ensuremath{\left\lvert #1 \right\rvert}}

\newcommand{\vpren}[1]{\ensuremath{\left( #1 \right)}}

\newcommand{\norm}[1]{\ensuremath{\lVert #1 \rVert}}
\newcommand{\abs}[1]{\ensuremath{\lvert #1 \rvert}}
\newcommand{\angles}[1]{\ensuremath{\langle #1 \rangle}}
\newcommand{\pren}[1]{\ensuremath{( #1 )}}

\expandafter\mathchardef\expandafter\varphi\number\expandafter\phi\expandafter\relax
\expandafter\mathchardef\expandafter\phi\number\varphi

\newcommand{\im}{\ensuremath{\mathrm{i}}}
\newcommand{\e}{\ensuremath{\mathrm{e}}}

\renewcommand{\tilde}{\widetilde}

\newcommand{\N}{\mathbb{N}}

\newcommand{\Hil}{\mathcal{H}}
\newcommand{\Kil}{\mathcal{K}}

\newcommand\iprec{\mathrel{\ooalign{$\prec$\cr
 \,\raise0.85ex\hbox{\scriptsize$\circ$}\cr}}}

\renewcommand{\emptyset}{\varnothing}
\renewcommand{\epsilon}{\varepsilon}

\DeclareMathOperator{\ran}{ran}

\DeclareMathOperator{\diag}{diag}
\DeclareMathOperator{\trace}{Tr}
\DeclareMathOperator{\spans}{span}

\setlist[enumerate]{label=(\roman*)}

\numberwithin{equation}{section}

\theoremstyle{plain} 
\newtheorem{theorem}{Theorem}[section]
\newtheorem{corollary}[theorem]{Corollary}
\newtheorem{lemma}[theorem]{Lemma}
\newtheorem{proposition}[theorem]{Proposition}
\newtheorem{question}[theorem]{Question}

\theoremstyle{definition}
\newtheorem{definition}[theorem]{Definition}

\theoremstyle{remark}
\newtheorem{remark}[theorem]{Remark}

\newtheorem{notation}[theorem]{Notation}

\title{Thompson's theorem for compact operators and diagonals of unitary operators}

\author{John Jasper}
\email[John Jasper]{john.jasper@uc.edu}
\author{Jireh Loreaux}
\email[Jireh Loreaux]{loreaujy@mail.uc.edu}
\author{Gary Weiss$^{*}$}
\email[Gary Weiss]{gary.weiss@uc.edu}
\thanks{$^{*}$Partially supported by 
Simons Foundation Collaboration Grant for Mathematicians \#245014 and the Charles Phelps Taft Research Center.}

\begin{document}

\begin{abstract}
  As applications of Kadison's Pythageorean and carpenter's theorems, the Schur--Horn theorem, and Thompson's theorem, we obtain an extension of Thompson's theorem to compact operators and use these ideas to give a  characterization of diagonals of unitary operators.
  Thompson's mysterious inequality concerning the last terms of the diagonal and singular value sequences plays a central role.
\end{abstract}


\maketitle

\section{Introduction}

The last century, especially the past 15 years, saw significant advances toward characterizing the diagonal sequences of various types of operators and classes of operators.
That is, given an operator $A$ (or a class of operators $\mathcal{A}$), the goal is to classify the sequences of its inner products $\pren{\angles{Ae_j,e_j}}_{j=1}^{\infty}$ for all orthonormal bases $\mathfrak{e} = \{e_j\}_{j=1}^{\infty}$ (for $A \in \mathcal{A}$).
Equivalently, given an operator $A$ and a fixed orthonormal basis $\mathfrak{e}$, identify all sequences $\pren{\angles{UAU^{*}e_j,e_j}}_{j=1}^{\infty}$ as $U$ ranges over all unitary operators (i.e., identify the image under the canonical trace-preserving conditional expectation of its unitary orbit).
This turned out to provide tools for also characterizing diagonal sequences for various important classes of operators (i.e., the expectation of unitary orbits of classes).
These questions grow out of Schur \cite{Sch-1923-SBMG} and Horn \cite{Hor-1954-AJM} whose combined work completely characterized the diagonal sequences of a selfadjoint matrix in $M_N(\mathbb{C})$ in terms of its eigenvalue sequence (see \autoref{thm:schur-horn}).
Moreover, Horn's proof yields the same result over $M_N(\mathbb{R})$ and we also include that here.

Note that an operator $A$ acting on a complex Hilbert space $\Hil_{\mathbb{C}}$ with a real-valued matrix representation in a basis $\mathfrak{e}$ restricts to an operator $A_{\mathbb{R}}$ acting on the real Hilbert space $\Hil_{\mathbb{R}} := \overline{\spans_{\mathbb{R}} \mathfrak{e}}$ with the same matrix representation.
Conversely, an operator $A_{\mathbb{R}}$ acting on $\Hil_{\mathbb{R}}$ extends naturally by linearity to the complexification $\Hil_{\mathbb{C}}$ and retains its matrix representation.
Thus for problems whose solutions depend on the existence of a specified matrix representation, finding a representation with real-valued entries is equivalent to solving the problem over a real Hilbert space, a convenient recurrent theme in this paper.

\begin{theorem}[Schur--Horn theorem \protect{\cite{Hor-1954-AJM,Sch-1923-SBMG}}]
  \label{thm:schur-horn}
  There is an $N\times N$ selfadjoint matrix in $M_N(\mathbb{C})$ (or even $M_N(\mathbb{R})$) with eigenvalue sequence $\bm{\uplambda}$ (i.e., in the unitary orbit of $\diag \lambda$) and diagonal $\mathbf{d}$ in $\mathbb{R}^N$ if and only if, for their nonincreasing rearrangements $\mathbf{d}^{*}, \bm{\uplambda}^{*}$,
  \begin{equation*}
    \sum_{i=1}^k d^{*}_i \leq \sum_{i=1}^k \lambda^{*}_i \quad\text{for } k=1,\ldots,N,
  \end{equation*}
  with equality when $k = N$.
\end{theorem}

The Schur--Horn theorem, as it has come to be known, has inspired many types of extensions.
The earliest were probably due to Markus \cite{Mar-1964-UMN} and Gohberg and Markus \cite{GM-1964-MSN} who proved a version for selfadjoint trace-class operators without specifying the number of zeros on the diagonal.
In \cite{Neu-1999-JFA}, Neumann characterized the set of diagonals of any selfadjoint operator up to the closure in the $\ell^{\infty}$ norm.
However, it was soon understood that this particular characterization sometimes loses subtle information as identified by Kadison in \cite{Kad-2002-PNASU,Kad-2002-PNASUa} where he proved an infinite dimensional version of the Pythagorean theorem and its converse, which he referred to as the carpenter's theorem.
These theorems of Kadison completely describe the diagonals of projections and include a subtle integer condition when the diagonals accumulate summably at 0 and 1 (see \autoref{thm:kadison} for details).

\begin{theorem}[Pythagorean and carpenter's theorems \protect{\cite{Kad-2002-PNASU,Kad-2002-PNASUa}}]
  \label{thm:kadison}
  A sequence $\mathbf{d}$ is the diagonal of a projection if and only if it takes values in $[0,1]$ and for
  \begin{equation*}
    a := \sum_{d_j < \frac{1}{2}} d_j
    \quad\text{and}\quad
    b := \sum_{d_j \ge \frac{1}{2}} (1-d_j),
  \end{equation*}
  either
  \begin{enumerate}
  \item $a + b = \infty$ or
  \item $a + b < \infty$ and $a - b \in \mathbb{Z}$.
  \end{enumerate}
\end{theorem}

This topic has flourished within the last decade with contributions by Arveson and Kadison \cite{AK-2006-OTOAaA} (positive trace-class operators), Kaftal and Weiss \cite{KW-2010-JFA} (positive compact operators), Bownik and Jasper \cite{Jas-2013-JFA,BJ-2013-TAMS} (selfadjoint operators with finite spectrum), and Loreaux and Weiss \cite{LW-2015-JFA} (positive compact operators with nonzero kernel).
Moreover, the Schur--Horn theorem has extensions to von Neumann algebras first proposed by Kadison for projections in type $\mathrm{II}_1$ factors \cite{Kad-2002-PNASU} and by Arveson and Kadison for selfadjoint operators in $\mathrm{II}_1$ factors \cite{AK-2006-OTOAaA}.
Some of the work produced along these lines includes papers by Argerami and Massey \cite{AM-2007-IUMJ,AM-2008-JMAA,AM-2013-PJM} (a contractive version and approximations in both $\mathrm{II}_1$ and $\mathrm{II}_{\infty}$ factors), Bhat and Ravichandran \cite{BR-2014-PAMS} (selfadjoint operators with finite spectrum in $\mathrm{II}_1$ factors), Dykema, Fang, Hadwin and Smith \cite{DFHS-2012-IJM} (certain masa/factor pairs), as well as the as yet unpublished work of Ravichandran \cite{Rav-2012} (general von Neumann algebras) and Massey and Ravichandran \cite{MR-2014} (several commuting selfadjoint operators).

While there remains work to be done on the topic of selfadjoint operators, it must be noted that interest in diagonals extends to normal operators as well.
In fact, Horn's original reason for investigating the selfadjoint case was merely to provide a tool to access the diagonals of rotation, orthogonal and unitary matrices (i.e., elements of $SO(N), O(N)$ and $U(N)$ respectively) \cite[Theorems 8--11]{Hor-1954-AJM}.
An important point to make here is that Horn did not classify the diagonals of any individual matrix from these three classes but rather the diagonals of each entire class, that is, the union of the diagonals of the matrices in each class.
One of the goals in this paper is to extend Horn's result about the diagonals of the class of unitary operators \cite[Theorem 11]{Hor-1954-AJM} to the infinite dimensional setting culminating in our \autoref{thm:unitary-thompson-sufficiency} below which we state in full generality, although it easily reduces to the nonnegative case via \autoref{prop:nonnegative-diagonals-suffice}.

\begin{restatable*}[Diagonals of the class of unitary operators]{theorem}{unitary}
  \label{thm:unitary-thompson-sufficiency}
  A complex-valued sequence $\mathbf{d}$ is the diagonal of a unitary operator if and only if $\abs{\mathbf{d}}$ is bounded above by one and
  \begin{equation}
    \label{eq:unitary-thompson-condition}
    2 \vpren{1-\inf_{j \in \mathbb{N}} \abs{d_j}} \le \sum_{j \in \mathbb{N}} \big( 1-\abs{d_j} \big).
  \end{equation}
  Moreover, if $\mathbf{d}$ is real-valued then the same statement holds over real Hilbert space.
\end{restatable*}

Unitaries are a special class of normal matrices, but Horn observed in the $3 \times 3$ normal case that the set of diagonals is, in general, not convex.
This dashed hope of a straightforward generalization of the Schur--Horn theorem because of its equivalent formulation in which the diagonals are the convex hull of permutations of the eigenvalue sequence, as Horn states in \cite{Hor-1954-AJM}.
The $3 \times 3$ normal case was solved by Williams \cite{Wil-1971-JLMS2}, but subsequent work on diagonals of normal operators has stalled almost entirely.
There are three notable exceptions.
First, on a separable infinite dimensional Hilbert space, Arveson \cite{Arv-2007-PNASU} provided a necessary condition for a sequence to be a diagonal of a normal operator with finite spectrum that forms the vertices of a convex polygon.
Second, Kennedy and Skoufranis \cite{KS-2014} have obtained a result in $\mathrm{II}_1$ factors for diagonals (conditional expectation of the unitary orbit onto masas) of normal operators.
Finally, Massey and Ravichandran \cite{MR-2014} used their own work on multivariable Schur--Horn theorems to provide certain approximate results on diagonals of normal operators in Type I factors by considering appropriate dilations of the algebra.

In spite of these difficulties encountered for normal operators, there has been progress in other directions by studying classes of operators instead of single operators and not restricting the operators in these classes to be normal.
For example, the work of Fong \cite{Fon-1986-PEMS2} shows that the diagonals of the class of nilpotent operators consists of all bounded sequences, while Loreaux and Weiss \cite{LW-2014} prove the same result for idempotent operators and moreover that the diagonals of finite rank idempotent operators consist precisely of those absolutely summable sequences whose sum is a positive integer (necessarily equal to the rank).
Additionally, the results showing that the class of nilpotents and the class of idempotents admit all bounded sequences as diagonals can be obtained as corollaries of the so-called pinching theorem due to Bourin \cite{Bou-2003-JOT}. The pinching theorem also provides information about some of the diagonals of a specified operator whenever its essential numerical range has nonempty interior.

Another finite dimensional result in this line of investigation is especially interesting because of its similarity to the Schur--Horn theorem.
This is due to Thompson \cite[Theorem~1 and Corollary~1]{Tho-1977-SJAM} and independently for dimension 2 to Sing \cite{Sin-1976-CMB}.
Thompson's theorem characterizes the diagonals of the class of operators with specified singular value sequence instead of specified eigenvalue sequence, as in the Schur--Horn theorem.

\begin{theorem}[Thompson \protect{\cite{Tho-1977-SJAM}}]
  \label{thm:thompson}
  Let $\mathbf s=\pren{s_{i}}_{i=1}^{N}$ be a nonincreasing sequence and $\mathbf d = \pren{d_{i}}_{i=1}^{N}$ a complex-valued sequence.
  There is an $N\times N$ matrix $A$ with singular value sequence $\mathbf s$ and diagonal $\mathbf d$ if and only if for the monotone nonincreasing rearrangment $\abs{\mathbf{d}}^{*} = (\abs{d}_1^{*}, \ldots, \abs{d}_N^{*})$ of the sequence of moduli of $\mathbf{d}$,
  \begin{equation*}
    \sum_{i=1}^k \abs{d}_i^{*} \leq \sum_{i=1}^k s_{i} \quad\text{for }k=1,\ldots,N
  \end{equation*}
  and
  \begin{equation*}
    \sum_{i=1}^{N-1} \abs{d}_i^{*} - \abs{d}_N^{*} \leq \sum_{i=1}^{N-1} s_i-s_N.
  \end{equation*}
  Moreover, if $\mathbf{d}$ is real-valued, we may choose the matrix $A$ to have real-valued entries.
\end{theorem}

\begin{remark}
  \label{rem:orthogonal-thompson}
  Thompson's theorem may be viewed in two ways: as a characterization of diagonals of operators with specified singular value sequence, or as a characterization of diagonals of the operators $U(\diag \mathbf{s})V$ as $U,V$ range over all unitary operators.
  The reader may notice that this is due to the fact, arising from the singular value decomposition, that operators of the form $UAV$ with $U,V$ unitary are precisely those that preserve the singular value sequence of $A$.
  That is, any operator which shares the singular values of $A$ and dimensionality of kernel and range can be expressed as a triple product in this way.
  The additional fact that if the desired diagonal $\mathbf{d}$ is real-valued then the matrix $A$ may be chosen to lie in $M_N(\mathbb{R})$ amounts to the equivalent statement that $U,V$ from $U(\diag \mathbf{s})V$ may be chosen to have real entries, which is a consequence of the singular value decomposition over $M_N(\mathbb{R})$.
  The Schmidt decomposition is an analogue of the singular value decomposition for compact operators.
\end{remark}

In view of the interest in normal operators, a natural question is whether or not the $N \times N$ matrix $A$ in Thompson's theorem can be chosen to be normal.
In general, this is false even for $2 \times 2$ matrices with distinct singular values.
Indeed, for a $2 \times 2$ normal matrix $A$, the singular values are simply the absolute values of the eigenvalues ($s_i = \abs{\lambda_i}$), but $(0,0)$ is a diagonal of $A$ if and only if $0 = \trace(A) = \lambda_1 + \lambda_2$, which means $s_1 = s_2$, but the zero sequence always satisfies Thompson's inequalities.
There is a host of open questions in this subject which are natural to explore.
In \autoref{sec:open-questions} we provide a partial list.

Of course, it is natural to ask how Thompson's theorem can be extended to infinite dimensions.
At first glance it may seem like a hopeless endeavor because the diagonals and singular values have no final, or necessarily even smallest, element.
However, for this reason we were led to consider in \autoref{sec:compact-thompson} compact operators where diagonal sequences and singular value sequences can always be placed in nonincreasing order converging to zero.
Intuitively, the occurrence of $d_N$ and $s_N$ in the final inequality of Thompson's theorem might be replaceable with zero, thus making it a redundant condition.
We prove exactly this in \autoref{thm:compact-thompson}.

\begin{restatable*}[Thompson's theorem for compact operators]{theorem}{thompson}
  \label{thm:compact-thompson}
  If $\mathbf s=\pren{s_{i}}_{i=1}^{\infty}$ is a nonnegative nonincreasing sequence and $\mathbf d = \pren{d_{i}}_{i=1}^{\infty}$ is a complex-valued sequence, both tending to zero, then there is a compact operator $A$ with singular value sequence $\mathbf s$ and diagonal $\mathbf d$ if and only if
  \begin{equation*}
    \sum_{i=1}^k \abs{d_i} \leq \sum_{i=1}^k s_i \quad\text{for } k \in \mathbb{N}.
  \end{equation*}
  Moreover, if $\mathbf{d}$ is real-valued then the statement holds over real Hilbert space.
\end{restatable*}

Our aforementioned \autoref{thm:unitary-thompson-sufficiency} characterization of diagonals of unitary operators includes a nontrivial condition which may be formally realized as a version of the final inequality in Thompson's theorem (see discussion immediately preceding \autoref{thm:unitary-thompson-necessity}).

\section{Background and notation}

\begin{notation}
  \label{not:background-notation}
  Let $c_0$ denote the set of (complex-valued) countably infinite sequences which converge to zero and $c_0^+$ those with nonnegative values.
  Within $c_0^+$ let $c_0^{*}$ denote those sequences which are nonincreasing.
  When sequences are denoted with a single letter they will be boldface and upright, either greek or roman letters.
  Otherwise sequences are listed between parentheses as $\mathbf{d} = (d_1,d_2,\ldots)$ or $\mathbf{d} = (d_1,\ldots,d_N)$, or more succinctly $\mathbf{d} = (d_i)_{i=1}^N$ where $N$ can be either finite or infinite.
  Let $\abs{\mathbf{d}} := \pren{\abs{d_i}}_{i=1}^{N}$.
  For a nonnegative sequence $\mathbf{d}$ (either finite or converging to zero), let $\mathbf{d}^{*}$ denote the nonincreasing rearrangement of $\mathbf{d}$ defined by: the $i$\textsuperscript{th} term of $\mathbf{d}^{*}$, $d^{*}_i$, is the $i$\textsuperscript{th} largest term of $\mathbf{d}$, respecting multiplicity.
  So when $\mathbf{d} \in c_0^+$ has infinite support, $\mathbf{d}^{*} > 0$ (i.e., $d_i > 0$ for all $i$) even if $\mathbf{d}$ is not.
  Although commonly used for this, the label ``nonincreasing rearrangement'' can be misleading.
  It is precise when $\mathbf{d}$ has finite support or when $\mathbf{d}$ is strictly positive, but when $\mathbf{d}$ has infinite support and any zeros there is no bijection $\pi$ of $\mathbb{N}$ for which $(d_{\pi(i)})_{i=1}^{\infty}$ is nonincreasing.

  We will often consider the direct sum $\mathbf{d_1} \oplus \mathbf{d_2}$ of two sequences $\mathbf{d_1}$ and $\mathbf{d_2}$, by which we mean any sequence which contains the elements of both $\mathbf{d_1}$ and $\mathbf{d_2}$ repeated according to multiplicity.
  The sequences may be either finite or infinite and order in $\mathbf{d_1} \oplus \mathbf{d_2}$ is irrelevant.
  The order of the direct sum sequence is not significant here because the class of diagonals of an operator is invariant under permutations.

  The inner product on a Hilbert space $\Hil$ is denoted by $\angles{\cdot,\cdot}$.
  For vectors $v,w \in \Hil$, let $v \otimes w$ denote the rank-one operator $x \mapsto \angles{x,w}v$.
  Operators in $B(\Hil)$ will be denoted with uppercase roman letters, but we will sometimes also use this typeface for special constants such as the underlying dimension or the length of a sequence.

  If $\mathcal{A}$ is the set of diagonal operators with respect to a fixed orthonormal basis basis $\mathfrak{e}$, let $\diag : \ell^{\infty} \to \mathcal{A}$ denote the canonical $*$-isomorphism given by
  \begin{equation*}
    \diag \mathbf{d} := \sum_{e \in \mathfrak{e}} d_i (e \otimes e).
  \end{equation*}
  When the basis is not explicitly specified it should be easily deduced from context.
\end{notation}

To avoid ambiguity, below is an explicit definition of singular values.

\begin{definition}
  Let $A$ be a compact operator on a Hilbert space $\Hil$.
  The \textit{singular values} of $A$ are the square roots of the eigenvalues of $A^{*}A$ which forms a sequence in $c_0^+$.
  Let $s_{i}(A)$ denote the $i$\textsuperscript{th} largest singular value of $A$ counting multiplicity.
  Let $s(A)$ denote the \emph{singular value sequence} $(s_{1}(A),s_{2}(A),\ldots)$.
\end{definition}

\begin{remark}
  Note that if $A$ has infinitely many positive singular values with or without a nontrivial kernel, then $\pren{s_{i}(A)}_{i=1}^{\infty}$ is a strictly positive sequence.
  That is, when $A$ is infinite rank the sequence $s(A)$ includes only the positive singular values of $A$.
\end{remark}

Similarly, to prevent confusion regarding the term \emph{compression} we provide a definition.

\begin{definition}
  Given an operator $A$ acting on $\Hil$ and a subspace $\Kil$, the \emph{compression of $A$ to $\Kil$} is the operator $PAP^{*} \in B(\Kil)$ where $P$ is the projection $P : \Hil \to \Kil$.
  Note that here $P^{*}$ is the adjoint as an operator between different Hilbert spaces, and in this case is equal to the inclusion map $P^{*} : \Kil \hookrightarrow \Hil$.
\end{definition}

We now provide definitions for the various notions of majorization we will use herein.

\begin{definition}
  \label{def:majorization}
  Given nonnegative nonincreasing sequences $\mathbf{d} = \pren{d_i}_{i=1}^N$ and $\mathbf{s} = \pren{s_i}_{i=1}^N$ for $N \in \mathbb{N} \cup \{\infty\}$, we say that $\mathbf{d}$ is \emph{weakly majorized} by $\mathbf{s}$, denoted $\mathbf{d} \prec_w \mathbf{s}$, if
  \begin{equation*}
    \sum_{i=1}^k d_i \le \sum_{i=1}^k s_i
    \quad\text{for}\ 1 \le k \le N.
  \end{equation*}
  When there is equality for $k=N$, we say that $\mathbf{d}$ is \emph{majorized} by $\mathbf{s}$, denoted $\mathbf{d} \prec \mathbf{s}$.
  If $N < \infty$ and $\mathbf{d} \prec_w \mathbf{s}$ and in addition
  \begin{equation*}
    \sum_{i=1}^{N-1} d_i - d_N \le \sum_{i=1}^{N-1} s_i - s_N,
  \end{equation*}
  then we say $\mathbf{d}$ is \emph{Thompson majorized} by $\mathbf{s}$, denoted $\mathbf{d} \prec_T \mathbf{s}$.
\end{definition}

We repeatedly use the following generalization of the Schur--Horn theorem to positive compact operators in infinite dimensions due to Kaftal and Weiss \cite[Proposition~6.6]{KW-2010-JFA}.

\begin{theorem}[Schur--Horn theorem for positive compact operators]
  \label{thm:compact-schur-horn}
  Given sequences $\mathbf{d}, \mathbf{s} \in c_0^{*}$, there is a positive compact operator with eigenvalue sequence $\mathbf{s}$ and diagonal $\mathbf{d}$ if and only if $\mathbf{d} \prec \mathbf{s}$.
  Moreover, the operator may be chosen to have real-valued entries in the basis in which it has diagonal $\mathbf{d}$.
\end{theorem}

The restriction of \autoref{thm:compact-schur-horn} to finite rank positive operators has been proven by several groups including, but almost certainly not limited to, Arveson and Kadison \cite[Proposition~3.1 and Theorem~4.1]{AK-2006-OTOAaA} and Kaftal and Weiss \cite[Lemma~6.3 and Proposition~6.4]{KW-2010-JFA} and it follows as an easy corollary of Kadison's carpenter's theorem for rank-one projections \cite[Proposition~1]{Kad-2002-PNASU} in conjunction with the classical Schur--Horn theorem (\autoref{thm:schur-horn}).
However, to our knowledge, only Kaftal and Weiss address the possibility that the operator has real-valued entries.

\section{Thompson's theorem for compact operators}
\label{sec:compact-thompson}

Firstly we show in \autoref{cor:nonnegative-diagonals-suffice-s-numbers} that in this approach to extending Thompson's \autoref{thm:thompson} above to compact operators, we may assume without loss of generality that $\mathbf{d} \ge 0$, for which we need the following.

\begin{proposition}
  \label{prop:nonnegative-diagonals-suffice}
  A sequence $\mathbf{d}$ is a diagonal of an operator $A$ if and only if $\abs{\mathbf{d}}$ is a diagonal of $UA$ for some diagonal unitary operator $U$.
  Moreover for our choice of $U$, if $\mathbf{d}$ is real-valued and either $A$ or $UA$ has real-valued entries in that basis, then both do.
\end{proposition}

\begin{proof}
  Assume there is an operator $A$ with diagonal $\mathbf d$.
  For each $i$ define the modulus one complex number
  \begin{equation*}
    z_i :=
    \begin{cases}
      \frac{\abs{d_i}}{d_i} & \text{if $d_i \not= 0$} \\
      1                     & \text{if $d_i = 0$.}    \\
    \end{cases}
  \end{equation*}
  so then $z_id_i = |d_i|$.
  For the unitary $U := \diag \mathbf{z}$,  the operator $UA$ has diagonal $\abs{\mathbf{d}}$.

  For the converse, apply the diagonal unitary $U^{*}$ on the left of an operator with diagonal $\abs{\mathbf{d}}$.
  Moreover, if $\mathbf{d}$ is real-valued and if $A$ has real-valued entries then so does $UA$.
\end{proof}

\begin{corollary}
  \label{cor:nonnegative-diagonals-suffice-s-numbers}
  For sequences $\mathbf{d} \in c_0$ and $\mathbf{s} \in c_0^{*}$, there is a compact operator with diagonal $\mathbf{d}$ and singular value sequence $\mathbf{s}$ if and only if there is a compact operator with diagonal $\abs{\mathbf{d}}$ and singular value sequence $\mathbf{s}$.
\end{corollary}

\begin{proof}
  Apply \autoref{prop:nonnegative-diagonals-suffice} and note that $A$ and $UA$ have the same singular value sequence.
\end{proof}

The following proposition is originally due to Ky Fan \cite[Theorem 1]{Fan-1951-PNASUSA}.
Fan's result is actually significantly more general than stated below, but this is a commonly used simplification and is all that is needed for our purposes.
His proof restricted to the special case of \autoref{prop:fans-theorem} essentially amounts to using the Schmidt decomposition for compact operators and then two applications of the Cauchy--Schwarz inequality, along with straightforward inequality manipulation.

\begin{proposition}[Fan \protect{\cite{Fan-1951-PNASUSA}}]
  \label{prop:fans-theorem}
  If $\mathbf{d}$ is a diagonal of a compact operator $A$ with singular value sequence $s(A)$ then $\abs{\mathbf{d}}^{*} \prec_w s(A)$.
\end{proposition}

To extend Thompson's Theorem for finite matrices to infinite matrices, it is natural to consider the rank-one case.
\autoref{lem:rk1} is subsumed by \autoref{thm:compact-thompson-sufficiency}, but we include it because of the interesting proof technique and angle observation.

\begin{lemma}
  \label{lem:rk1}
  Let $\mathbf s = (s_1,0,0,\ldots), \mathbf d = (d_{1},d_{2},\ldots) \in c_0^{*}$ with $s_1 > 0$.
  There is a rank-one operator $A$ with singular value sequence $\mathbf s$ and diagonal $\mathbf d$ if and only if
  \begin{equation*}
    \sum_{i=1}^{\infty} d_i \leq s_{1}.
  \end{equation*}
\end{lemma}

\begin{proof}
  Suppose $\sum_{i=1}^{\infty} d_i \leq s_{1}$.
  We may assume $s_1=1$, since the general case follows by scaling.

  Let $\{e_{i}\}_{i=1}^{\infty}$ be an orthonormal basis.
  For each of three cases, we will define two sequences of nonnegative numbers $\pren{a_{i}}_{i=1}^{\infty}$ and $\pren{b_{i}}_{i=1}^{\infty}$, the corresponding vectors
  \begin{equation*}
    v = \sum_{i=1}^{\infty}a_{i}e_{i}\quad\text{and}\quad w=\sum_{i=1}^{\infty}b_{i}e_{i},
  \end{equation*}
  and the operator $A = v \otimes w$, i.e., $Af = \langle f,w\rangle v$.
  A simple calculation shows that the singular value sequence of $A$ is $(\norm{v}\norm{w},0,0,\ldots)$, and the diagonal of $A$ is $\pren{a_{i}b_{i}}_{i=1}^{\infty}$.

  \emph{Case 1: $\mathbf d=0$.}
  Set $v=e_{1}$ and $w=e_{2}$.
  In this case we have $\norm{v}=\norm{w}=1$ and $a_{i}b_{i}=0=d_{i}$ for all $i$.

  \emph{Case 2: $d_{1}>d_{2}=0$.}
  Set $v=\sqrt{d_{1}}e_{1}+\sqrt{1-d_{1}}e_{2}$ and $w=\sqrt{d_{1}}e_{1}+\sqrt{1-d_{1}}e_{3}$, so then $\norm{v}=\norm{w}=1$ and $a_{1}b_{1}=d_{1}$ and $a_{i}b_{i}=0=d_{i}$ for all $i\geq 2$.

  \emph{Case 3: $d_{2}>0$.}
  For each $\alpha>0$ set
  \begin{equation*}
    v_{\alpha} = \alpha\sqrt{d_{1}}e_{1} + \sum_{i=2}^{\infty}\sqrt{d_{i}}e_{i}\quad\text{and}\quad w_{\alpha} = \frac{1}{\alpha}\sqrt{d_{1}}e_{1} + \sum_{i=2}^{\infty}\sqrt{d_{i}}e_{i}
  \end{equation*}
  For any choice of $\alpha$ we see that $a_{i}b_{i} = d_{i}$ for all $i\in\N$.
  We calculate
  \begin{equation*}
    \norm{v_{\alpha}}^{2}\norm{w_{\alpha}}^{2} = d_{1}^2 + \left(\alpha^2+\frac{1}{\alpha^2}\right)d_{1}\sum_{i=2}^{\infty}d_{i} + \left(\sum_{i=2}^{\infty}d_{i}\right)^{2}.
  \end{equation*}
  When $\alpha=1$ we have
  \begin{equation*}
    \norm{v_{1}}^{2}\norm{w_{1}}^{2} = d_{1}^2 + 2d_{1}\sum_{i=2}^{\infty}d_{i} + \left(\sum_{i=2}^{\infty}d_{i}\right)^{2} = \left(\sum_{i=1}^{\infty}d_{i}\right)^{2}\leq s_{1}^{2}.
  \end{equation*}

  It is clear that $\norm{v_{\alpha}}\norm{w_{\alpha}}$ is continuous for $\alpha\in(0,\infty)$.
  Since $d_{1},d_{2}>0$ we see that $\norm{v_{\alpha}}\norm{w_{\alpha}}\to\infty$ as $\alpha\to\infty$.
  Thus, for some $\beta>0$ we have
  \begin{equation*}
    \norm{v_{\beta}}\norm{w_{\beta}} = s_{1}.
  \end{equation*}
  Setting $A = v_{\beta} \otimes w_{\beta}$ gives the desired result.

  The converse is clear from \autoref{prop:fans-theorem}.
\end{proof}

\begin{remark}
  In this rank-one case, among all solutions $A$ the angle between $\ker^{\perp} A$ and $\ran A$ is unique and determined by $s_1 = \norm{A}$ and $\trace A$.
  In fact, if $A = v \otimes w$ is any rank-one operator and $\theta$ is the angle between $v,w$, then
  \begin{equation*}
    \cos \theta = \frac{\abs{\angles{v,w}}}{\norm{v}\cdot\norm{w}} = \frac{\abs{\trace A}}{\norm{A}}.
  \end{equation*}
  Moreover, the angle between $\ker^{\perp} A$ and $\ran A$ is precisely the angle between $v$ and $w$.
\end{remark}

A natural next step would be to prove a finite rank Thompson's theorem such as \autoref{cor:finite-rank-thompson}.
However all of our proofs of this result had substantial overlap with the proof of \autoref{thm:compact-thompson}.
As the result we simply state the finite rank version as a corollary of \autoref{thm:compact-thompson}.

\begin{corollary}[Finite rank Thompson's theorem]
  \label{cor:finite-rank-thompson}
  Let $\mathbf s$ and $\mathbf d$ be nonincreasing nonnegative sequences with $\mathbf s$ of finite support.
  There is a finite rank operator $A$ with singular values $\mathbf s$ and diagonal $\mathbf d$ if and only if
  \begin{equation}
    \label{eq:frt0}
    \sum_{i=1}^{k}d_{i}\leq \sum_{i=1}^{k}s_{i} \quad\text{for all }k\in\N.
  \end{equation}
\end{corollary}

Although basic, the next lemma is a fundamental tool in the construction of diagonals but to our knowledge it has not yet appeared in the literature.
Herein we use it in the proofs of \autoref{thm:compact-thompson} case 3 and \autoref{thm:unitary-thompson-sufficiency}.

\begin{lemma}
  \label{lem:compression-diagonal}
  Suppose that $A$ is an operator acting on $\Hil$ and $\Kil$ is a subspace.
  If the compressions of $A$ to $\Kil$ and $\Kil^{\perp}$ have diagonals $\mathbf{d}_1$ and $\mathbf{d}_2$ respectively, then $A$ has diagonal $\mathbf{d}_1 \oplus \mathbf{d}_2$.
\end{lemma}

\begin{proof}
  By hypothesis, there is a basis $\mathfrak{e}_1$ for $\Kil$ with respect to which the compression of $A$ to $\Kil$ has diagonal $\mathbf{d}_1$.
  Similarly, there is a basis $\mathfrak{e}_2$ for $\Kil^{\perp}$ corresponding to $\mathbf{d}_2$.
  Let $P : \Hil \to \Kil$ and $P^{\perp} : \Hil \to \Kil^{\perp}$ be the standard projections.
  Set $\mathfrak{e} := \mathfrak{e}_1 \cup \mathfrak{e}_2$ and notice that if $e \in \mathfrak{e}_1$, then $Pe = e$, and therefore $\angles{Ae,e}_{\Hil} = \angles{AP^{*}e,P^{*}e}_{\Hil} = \angles{PAP^{*}e,e}_{\Kil}$.
  Similarly, if $e \in \mathfrak{e}_2$ then $\angles{Ae,e}_{\Hil} = \angles{A(P^{\perp})^{*}e,(P^{\perp})^{*}e}_{\Hil} = \angles{P^{\perp}A(P^{\perp})^{*}e,e}_{\Kil}$, and hence $A$ has diagonal $\mathbf{d}_1 \oplus \mathbf{d}_2$ with respect to $\mathfrak{e}$.
\end{proof}

\begin{theorem}
  \label{thm:compact-thompson-sufficiency}
  If $\mathbf{s} \in c_0^{*}$ and $\mathbf{d} \in c_0^+$ are sequences with $\mathbf{d}^{*} \prec_w \mathbf{s}$, then $\mathbf{d}$ is a diagonal of a compact operator $A$ with singular value sequence $\mathbf{s}$.
  Moreover, we may choose $A$ to have real-valued entries with diagonal $\mathbf{d}$.
\end{theorem}

\begin{proof}
  We begin by reducing to the case when $\mathbf{d} \in c_0^{*}$.
  Let $\mathbf{d} \in c_0^+$ and notice that we may write $\mathbf{d} = \mathbf{d}' \oplus \mathbf{0}_m$ where $\mathbf{d}' \in c_0^{*}$ for some $m \in \mathbb{Z}_{\ge 0} \cup \{\infty\}$ (if $\mathbf{d}$ has finite support choose $m=0$, otherwise choose $m = \abs{d^{-1}(0)}$).
  Suppose that there is some compact operator $A$ with diagonal $\mathbf{d}'$ and singular value sequence $s(A) = \mathbf{s}$.
  Then certainly $\mathbf{d} = \mathbf{d}' \oplus \mathbf{0}_m$ is a diagonal of $A \oplus 0_m$ which satisfies $s(A \oplus 0_m) = s(A) = \mathbf{s}$.
  Moreover, if $A$ has real-valued entries, so does $A \oplus 0_m$.
  Therefore we may assume without loss of generality that $\mathbf{d} \in c_0^{*}$.

  As a matter of notation, throughout the remainder of the proof denote by $\bm{\updelta}$ the sequence $\pren{\delta_n}_{n=1}^{\infty}$ whose terms are given by $\delta_n = \sum_{j=1}^n (s_j-d_j)$.
  Note that $\bm{\updelta} \ge 0$ since $\mathbf{d} = \mathbf{d}^{*} \prec_w \mathbf{s}$.

  The rest of this proof has three cases, the last two of which are harder.

  \emph{Case 1: $\mathbf{d} \prec \mathbf{s}$}.
  Apply the Schur--Horn theorem for positive compact operators (\autoref{thm:compact-schur-horn}) which also guarantees $A$ can be chosen to have real-valued entries.

  Note that this case includes the situations both when $\liminf \bm{\updelta} = 0$ and when $\mathbf{d},\mathbf{s} \notin \ell^1$, so that one has majorization (not merely weak majorization) because one has equality of the infinite sums in \autoref{def:majorization}.

  \emph{Case 2: $\liminf \bm{\updelta} > 0$ and $\delta_n < \liminf \bm{\updelta}$ for infinitely many $n$.}
  Note that in this case, for each $k \in \mathbb{N}$, we have $\inf_{n>k} \delta_n < \liminf \bm{\updelta}$ and moreover this infimum is attained by finitely many indices $n > k$.

  Set $k_0 := 0$ and define $k_j$ inductively by letting $k_{j+1}$ be the largest index $m$ satisfying $\delta_m = \inf_{n>k_j} \delta_n$, in particular $\delta_{k_{j+1}} = \inf_{n>k_j} \delta_n$.
  Thus we necessarily have $\delta_{k_j} < \delta_n$ if $n > k_j$, and hence
  \begin{equation}
    \label{eq:k-j-property}
    \sum_{i=k_j+1}^n (s_i-d_i) = \delta_n - \delta_{k_j} > 0.
  \end{equation}
  Since $\mathbf{d} \in c_0^{*}$, for each $j \in \mathbb{N}$ there exist distinct $m_j \ge k_j + j$ for which
  $d_{m_j} < \min \{ \delta_n - \delta_{k_j} \mid k_j+1 \le n \le k_{j+1} \}$, and therefore for $k_j+1 \le n \le k_{j+1}$,
  \begin{equation*}
    \sum_{i=k_j+1}^n (s_i-d_i) - d_{m_j} > 0.
  \end{equation*}

  We next partition $\mathbb{N}$ inductively.
  Consider $N_1 = \{ 1,\ldots,k_1,m_1 \}$ and define disjoint $N_{j+1}$ inductively as the smallest $k_{j+1}-k_j$ elements of $\mathbb{N} \setminus \left( \bigcup_{i=1}^j N_j \right)$ along with $m_{j+1}$.
  By our choice of $m_{j+1}$ we show that $m_{j+1} \notin \bigcup_{i=1}^j N_i$ and there are at least $k_{j+1}-k_j$ smaller elements in $\mathbb{N} \setminus \left( \bigcup_{i=1}^j N_i \right)$.
  Indeed, the number of elements in this latter set which are less than $m_{j+1}$ is minimized when all the elements of $\bigcup_{i=1}^j N_i$ (of which there are $\sum_{i=1}^{j-1} (k_{i+1} - k_i + 1) = k_j + j$) are all less than $m_{j+1}$.
  Since $m_{j+1} \ge k_{j+1} + j+1$ there are at least $k_{j+1} - k_j$ elements of $\mathbb{N} \setminus \left( \bigcup_{i=1}^j N_i \right)$ which are strictly smaller than $m_{j+1}$.
  A straightforward argument by induction then establishes $\bigcup_{i=1}^j N_i \subseteq \{1,\ldots,k_j+j\} \cup \{m_1,\ldots,m_j\}$ and hence $m_{j+1} \notin \left( \bigcup_{i=1}^j N_i \right)$ since $m_{j+1} \ge k_{j+1} + j + 1$ and since we chose the $m_i$ to be distinct.

  Define
  \begin{equation*}
    \mathbf{d}^j := \pren{d_{\phi_j(1)},\ldots,d_{\phi_j(k_j-k_{j-1}+1)}}
    \quad\text{and}\quad
    \mathbf{s}^j := \pren{s_{k_{j-1}+1},\ldots,s_{k_j},0},
  \end{equation*}
  where $\phi_j : \{1,\ldots,k_j-k_{j-1}+1 \} \to N_j$ is the order preserving bijection.

  Note that $\{1,\ldots,k_j\} \subseteq \bigcup_{i=1}^{j-1} N_i$ and therefore $\{1,\ldots,k_j\} \cap N_{j+1} = \emptyset$.
  Along with the fact that $\mathbf{d}$ is nonincreasing this implies for $1 \le n \le k_j - k_{j-1}$,
  \begin{equation*}
    \sum_{i=1}^n d^j_i = \sum_{i=1}^n d_{\phi_j(i)} \le \sum_{i=1}^n d_{k_{j-1}+i} = \sum_{i=k_{j-1}+1}^{k_{j-1}+n} d_i.
  \end{equation*}
  Combining this with equation \eqref{eq:k-j-property} yields
  \begin{equation*}
    \sum_{i=1}^n (s^j_i - d^j_i) \ge \sum_{i=k_{j-1}+1}^{k_{j-1}+n} (s_i - d_i) = \delta_{k_{j-1}+n} - \delta_{k_{j-1}} > 0,
  \end{equation*}
  over these same values of $n$.
  The choice of $m_j$ guarantees
  \begin{equation*}
    \sum_{i=1}^{k_j-k_{j-1}+1} (s^j_i - d^j_i) \ge \delta_{k_j} - \delta_{k_{j-1}} - d_{m_j} > 0,
  \end{equation*}
  and hence $\mathbf{d}^j \prec_w \mathbf{s}^j$.
  Finally, since the last term of $\mathbf{s}^j$ is zero, the final inequality for Thompson majorization is trivially satisfied.

  Therefore, by Thompson's theorem (\autoref{thm:thompson}) there is a matrix $A_j \in M_{k_j-k_{j-1}+1}(\mathbb{C})$ with real-valued entries and diagonal $\mathbf{d}^j$ such that $s(A_j) = \mathbf{s}^j$.
  Finally, letting $A = \bigoplus_{j=1}^{\infty} A_j$ we find that $A$ has real-valued entries, diagonal $\mathbf{d} = \oplus_j \mathbf{d}^j$ and singular value sequence $s(A) = (\oplus_j s(A_j))^{*} = (\oplus_j \mathbf{s}^j)^{*} = \mathbf{s}$.

  \emph{Case 3: Eventually $\delta_n \ge \lim \bm{\updelta} > 0$}.
  Note that we assumed that $\bm{\updelta}$ is convergent.
  This is not an additional assumption because the case when $\mathbf{s} \notin \ell^1$ is already handled by the previous two cases.
  In particular, if $\mathbf{d},\mathbf{s} \notin \ell^1$, then $\mathbf{d} \prec \mathbf{s}$, and if $\mathbf{d} \in \ell^1$ but $\mathbf{s} \notin \ell^1$ then we are in Case 2.
  Therefore, we may assume now that $\mathbf{d},\mathbf{s} \in \ell^1$ and hence $\bm{\updelta}$ is convergent.

  There are now two subcases.
  The first subcase is that $\bm{\updelta}$ is eventually constant, which is equivalent to saying that $\mathbf{d},\mathbf{s}$ have identical tails.
  In this case, apply the finite Thompson's \autoref{thm:thompson} to an initial segment $\mathbf{d}', \mathbf{s}'$ of the sequences $\mathbf{d}, \mathbf{s}$ which terminates \emph{after} the terms become identical.
  Clearly Thompson's theorem applies because weak majorization is guaranteed by hypothesis and the last terms in these finite sequences are the same, so the final inequality in Thompson majorization is satisfied.
  Thus there is a finite matrix $A_1$ with real-valued entries, singular value sequence $\mathbf{s}'$ and diagonal $\mathbf{d}'$.
  Because the remainders $\mathbf{d}'', \mathbf{s}''$ of the sequences $\mathbf{d}, \mathbf{s}$ are identical, the operator $A_2 := \diag \mathbf{s}'' = \diag \mathbf{d}''$ suffices for this portion of the sequences.
  Hence $A := A_1 \oplus A_2$ has real-valued entries, diagonal $\mathbf{d}$ and singular value sequence $s(A) = s(A_1 \oplus A_2) = \mathbf{s}$.

  The second subcase is the one where $\bm{\updelta}$ is not eventually constant, which means that $\delta_n > \lim \bm{\updelta}$ for infinitely many $n$.
  In this case, choose $k > 1$ large enough so that $\delta_n \ge \lim \bm{\updelta}$ for $n \ge k-1$.
  Moreover, since $\delta_n > \lim \bm{\updelta}$ for infinitely many $n$, $\mathbf{d}$ has infinite support and so we can ensure that $d_{k-1} > d_k$ (by possibly choosing a larger $k$).
  Then choose $k' > k$ so that $(\delta_{k'} - \lim \bm{\updelta}) \le \min\{ d_{k-1}-d_k, \lim \bm{\updelta} \}$ and also $\delta_{k'-1} \ge \delta_{k'}$.
  The first condition ensures $a := d_k + \delta_{k'} - \lim \bm{\updelta} \le d_{k-1}$ and $\delta_{k'} - \lim \bm{\updelta} \le \lim \bm{\updelta}$, whereas the second condition is equivalent to $d_{k'} \ge s_{k'}$.

  Now consider $\mathbf{s}' := \pren{s_1,\ldots,s_{k'}}$ and $\mathbf{d}' := \pren{d_1,\ldots,d_{k-1},a,d_{k+1},\ldots,d_{k'}}$.
  These are both nonincreasing since $\mathbf{d},\mathbf{s} \in c_0^{*}$ and $d_k \le a \le d_{k-1}$.
  We claim that $\mathbf{d}' \prec_w \mathbf{s}'$.
  To see this, note that for $n < k$ we have
  \begin{equation*}
    \sum_{j = 1}^n (s'_j - d'_j) = \sum_{j=1}^n (s_j - d_j) = \delta_n \ge 0.
  \end{equation*}
  For $n \ge k$ we have
  \begin{align*}
    \sum_{j = 1}^n (s'_j - d'_j) & = \sum_{j=1}^{k-1} (s_j - d_j) + (s_k - a) + \sum_{j=k+1}^n (s_j - d_j) \\
                                 & = \sum_{j=1}^n (s_j - d_j) -  ( \delta_{k'} - \lim \bm{\updelta} )                \\
                                 & \ge \delta_n - \lim \bm{\updelta} \ge 0,
  \end{align*}
  where the last inequality follows since $n \ge k$ and so $\mathbf{d} \prec_w \mathbf{s}$.

  Now consider $\mathbf{s}'' := \pren{a,s_{k'+1},s_{k'+2},\ldots}$ and $\mathbf{d}'' := \pren{d_k,d_{k'+1},d_{k'+2},\ldots}$.
  Both of these sequences are nonincreasing since $\mathbf{d},\mathbf{s} \in c_0^{*}$ and $a \ge d_k \ge d_{k'} \ge s_{k'} \ge s_{k'+1}$.
  We will show $\mathbf{d}'' \prec \mathbf{s}''$.
  Indeed,
  \begin{equation*}
    \sum_{j=1}^n (s''_j - d''_j) = a - d_k + \delta_{k'+n-1} - \delta_{k'} = \delta_{k'+n-1} - \lim \bm{\updelta} \ge 0.
  \end{equation*}
  Moreover, taking the limit as $n \to \infty$ attains zero, which means $\mathbf{d}'' \prec \mathbf{s}''$.

  Because $\mathbf{d}' \prec_w \mathbf{s}'$ and $s_{k'} \le d_{k'}$ we have  $\mathbf{d}' \prec_T \mathbf{s}'$ and so we can apply Thompson's theorem to obtain a $k' \times k'$ matrix $A_1$ with real-valued entries acting on $\Hil_1$ with diagonal $\mathbf{d}'$ and singular values $\mathbf{s}'$.
  Let the basis corresponding to $\mathbf{d}'$ be denoted by $\mathfrak{e}_1 := \{e_j\}_{j=1}^{k'}$.
  Let $A_2 := \diag (s_{k'+1},s_{k'+2},\ldots)$ act on $\Hil_2$ with respect to the basis $\mathfrak{e}_2 := \{e_{k'+j}\}_{j=1}^{\infty}$.
  Then the operator $A := A_1 \oplus A_2$ has real-valued entries, singular value sequence $\mathbf{s}$ and the compression $\tilde{A}_2$ of $A$ to $\spans\{e_k,\Hil_2\}$ is $\diag \mathbf{s}''$ with respect to the basis $\{e_k\} \cup \mathfrak{e}_2$ for that subspace.
  Moreover, the compression $\tilde{A}_1$ of $A$ onto $\spans\{e_k,\Hil_2\}^{\perp}$ has diagonal $(d_1,\ldots,\hat{d}_k,\ldots,d_{k'})$ (where the hat indicates $d_k$ is omitted).
  Because $\mathbf{d}'' \prec \mathbf{s}''$ we can apply the Schur--Horn theorem for positive compact operators (\autoref{thm:compact-schur-horn}) to conclude that $\diag \mathbf{s}''$ has $\mathbf{d}''$ as a diagonal in some basis.
  Moreover, this change of basis can be achieved via an orthogonal matrix (unitary with real-valued entries relative to this basis $\mathfrak{e}_1 \cup \mathfrak{e}_1$), and so $A$ has real-valued entries in the resulting basis.
  Therefore, by \autoref{lem:compression-diagonal}, $A$ has diagonal $\mathbf{d} = (d_1,\ldots,\hat{d}_k,\ldots,d_{k'}) \oplus \mathbf{d}''$.
\end{proof}

Together, \autoref{prop:nonnegative-diagonals-suffice}, \autoref{prop:fans-theorem} and \autoref{thm:compact-thompson-sufficiency} prove directly Thompson's theorem for compact operators.
\autoref{prop:fans-theorem} proves the statement, and for the converse \autoref{prop:nonnegative-diagonals-suffice} reduces to the case $\mathbf{d} \ge 0$ and \autoref{thm:compact-thompson-sufficiency} yields the rest.

\thompson

\section{Diagonals of unitary operators}

Our approach starts with unitaries possessing a diagonal of special type.
The next lemma is a curious feature about operators with a diagonal whose entries are almost norm-attaining in a summable sense.
This leads to \autoref{thm:unitary-thompson-necessity} which places a necessary condition on the diagonals of unitary operators whose entries approach the unit circle summably.
It turns out that for sequences of this type, this necessary condition is also sufficient (see \autoref{thm:unitary-thompson-sufficiency}).

\begin{lemma}
  \label{lem:hs-perturb-identity}
  Let $A$ be a contraction with diagonal $\mathbf{d}$ with respect to the basis $\mathfrak{e} = \{e_i\}_{i=1}^{\infty}$.
  \begin{equation*}
    \text{If } \sum_{j=1}^{\infty} \big( 1-\abs{d_j}^2 \big) < \infty, \text{ then $A - \diag \mathbf{d}$ is Hilbert--Schmidt,}
  \end{equation*}
  and so also is $\diag \mathbf{u} - A$, where $u_j = \frac{d_j}{\abs{d_j}}$ if $d_j \not= 0$ and $u_j = 1$ otherwise.
  In addition, if $d_j \ge 0$, then $I - A$ is Hilbert--Schmidt.
  Moreover, whenever $\abs{d_j} = 1$, $e_j$ is an eigenvector.
\end{lemma}

\begin{proof}
  Let $\{e_j\}_{j=1}^{\infty}$ denote the basis corresponding to the diagonal $\mathbf{d}$.
  Since $\norm{A} \le 1$, we know
  \begin{equation*}
    1 \ge \norm{Ae_j}^2 = \sum_{i=1}^{\infty} \abs{ \angles{Ae_j,e_i} }^2.
  \end{equation*}
  Summing over $1 \le j \le k$, we find
  \begin{equation*}
    k \ge \sum_{j=1}^k \abs{d_j}^2 + \sum_{j=1}^k \sum_{\substack{i=1 \\ i \not= j}}^{\infty} \abs{ \angles{ Ae_j,e_i } }^2.
  \end{equation*}
  Rearranging and letting $k \to \infty$, we obtain
  \begin{equation*}
    \sum_{j=1}^\infty \big( 1-\abs{d_j}^2 \big) \ge \sum_{\substack{i,j=1 \\ i \not= j}}^{\infty} \abs{ \angles{ Ae_j,e_i } }^2 = \norm{A - \diag \mathbf{d}}_2^2,
  \end{equation*}
  which proves $A - \diag \mathbf{d}$ is Hilbert--Schmidt since the left-hand side is finite by hypothesis.
  To prove $\diag \mathbf{u} - A$ is Hilbert--Schmidt, it suffices to prove that $\diag \mathbf{u} - \diag \mathbf{d}$ is Hilbert--Schmidt.
  To see this, when $d_j \not= 0$ simply note that
  \begin{equation*}
    \vabs{\frac{d_j}{\abs{d_j}}-d_j}^2 = \big(1-\abs{d_j} \big)^2 \le  \big( 1-\abs{d_j} \big) \big( 1+\abs{d_j} \big) = 1-\abs{d_j}^2,
  \end{equation*}
  and when $d_j = 0$, $\abs{u_j - d_j}^2 = 1 = 1-\abs{d_j}^2$.
  Hence
  \begin{equation*}
    \sum_{j=1}^{\infty} \abs{u_j-d_j}^2 \le \sum_{j=1}^{\infty} \big( 1-\abs{d_j}^2 \big) < \infty,
  \end{equation*}
  from which the second claim follows.

  Finally, suppose $\abs{d_j} = 1$ for some $j$.
  By the Cauchy--Schwarz inequality
  \begin{equation*}
    1 = \abs{d_j} = \abs{ \angles{ Ae_j,e_j } } \le \norm{Ae_j} \cdot \norm{e_j} \le \norm{A} \le 1,
  \end{equation*}
  and since we have equality, $Ae_j = d_je_j$.
\end{proof}

Via \autoref{prop:nonnegative-diagonals-suffice} the next theorem places a necessary condition on certain diagonals of unitary operators.
This can be viewed as an analogue of the final inequality of Thompson's theorem (\autoref{thm:thompson}).
To see the correspondence, note that if in Thompson's theorem, instead of nonincreasing order, we arrange $\abs{\mathbf{d}}$ and $\mathbf{s}$ in \emph{nondecreasing} order, then the final inequality may be rewritten as:
\begin{equation}
  \label{eq:formal-thompson-inequality}
  s_1 - \abs{d_1} \le \sum_{i=2}^N \big( s_i - \abs{d_i} \big).
\end{equation}
Moreover, for unitary operators we have $s_i = 1$ for all $1 \le i \le N$.
Passing in \eqref{eq:formal-thompson-inequality} to the limit as $N \to \infty$ we formally obtain the necessary condition of \autoref{thm:unitary-thompson-necessity}.
The proof of this next theorem proceeds by establishing an $\epsilon$-approximate form of \eqref{eq:formal-thompson-inequality} using Thompson's theorem applied to a finite compression of the unitary operator followed by examining limiting behavior.

\begin{theorem}
  \label{thm:unitary-thompson-necessity}
  If $U$ is a unitary operator with nonnegative nondecreasing diagonal $\mathbf{d}$ for which $\sum_{j=1}^{\infty} (1-d_j)$ is finite, then
  \begin{equation*}
    1-d_1 \le \sum_{j=2}^{\infty} (1-d_j).
  \end{equation*}
\end{theorem}

\begin{proof}
  Note that since $d_i \le \norm{U} = 1$,
  \begin{equation*}
    \sum_{j=1}^{\infty} (1-d_j^2) = \sum_{j=1}^{\infty} (1-d_j)(1+d_j) \le 2 \sum_{j=1}^{\infty} (1-d_j) < \infty.
  \end{equation*}
  Thus $\sum_{j=1}^{\infty} (1-d_j^2) < \infty$ if and only if $\sum_{j=1}^{\infty} (1-d_j) < \infty$.
  Therefore, by \autoref{lem:hs-perturb-identity} and since $\mathbf{d} \ge 0$ we find that $I-U$ is Hilbert--Schmidt.
  Let $P_n$ denote the projection onto $\spans \{e_1,\ldots,e_n\}$, where $e_j$ is the basis element corresponding to the diagonal entry $d_j$.
  Let $A_n = P_nUP_n$ and $B_n = P_nUP_n^{\perp} = P_n(U-I)P_n^{\perp}$.
  Since $\norm{P_n-A_n}_2^2 = \norm{P_n(I-U)P_n}_2^2 \to \norm{I-U}_2^2$, we know that $\norm{B_n^{*}B_n}_1 = \norm{B_n}_2^2 =: \epsilon_n \to 0$.
  Since $U^{*}U = I$ we find that $A_n^{*}A_n + B_n^{*}B_n = P_n$.
  Rearranging, we find that $B_n^{*}B_n = P_n - A_n^{*}A_n \ge 0$, and so the eigenvalues of this latter operator are simply $1-(s_j(A_n))^2 \ge 0$ for $1 \le j \le n$.
  Taking the trace and using standard inequalities yields
  \begin{equation*}
    \sum_{j=1}^n \big( 1-s_j(A_n) \big) \le \sum_{j=1}^n \Big( 1- \big( s_j(A_n) \big)^2 \Big) = \trace (B_n^{*}B_n) = \epsilon_n,
  \end{equation*}
  and hence in particular $-s_n(A_n) \le -1 + \epsilon_n$.
  Finally, we apply the finite version of Thompson's theorem to $A_n$ and its diagonal sequence $(d_1,\ldots,d_n)$ to obtain
  \begin{equation*}
    \sum_{j=2}^n d_j - d_1 \le \sum_{j=1}^{n-1} s_j(A_n) - s_n(A_n) \le (n-1) - 1 + \epsilon_n,
  \end{equation*}
  or equivalently,
  \begin{equation*}
    -\epsilon_n \le \sum_{j=2}^n (1-d_j) - (1-d_1),
  \end{equation*}
  and taking the limit as $n \to \infty$ proves the desired inequality.
\end{proof}

\unitary

\begin{proof}
  By \autoref{prop:nonnegative-diagonals-suffice} we may without loss of generality restrict consideration to $\mathbf{d} \ge 0$.

  Suppose $\mathbf{d}$ is the diagonal of a unitary operator $U$.
  Then $d_i \le \norm{U} = 1$ for all $i \in \mathbb{N}$.
  When the sum in \eqref{eq:unitary-thompson-condition} is infinite there is nothing to prove for the implication.
  When the sum is finite the infimum is necessarily attained and we can relabel the diagonal entries so that this occurs at $d_1$. The necessity of condition \eqref{eq:unitary-thompson-condition} is then established by \autoref{thm:unitary-thompson-necessity}.

  For the converse, suppose that $0 \le \mathbf{d} \le 1$ and satisfies \eqref{eq:unitary-thompson-condition}.
  If the sum is infinite, we can apply Kadison's carpenter's theorem  (see \autoref{thm:kadison}) to the sequence $\frac{1}{2}(\mathbf{d+1})$ to get a projection $P$ with this as its diagonal.
  Then the symmetry (selfadjoint unitary) $U = 2P-I$ has diagonal $\mathbf{d}$.
  Moreover, Bownik and Jasper have shown in \cite{BJ-2014-CMB} that the projection $P$ can be chosen to have real-valued entries, and so the resulting unitary $U$ also has real-valued entries.

  Now suppose the sum in condition \eqref{eq:unitary-thompson-condition} is finite.
  As previously mentioned, the infimum is attained and there is no loss in assuming this occurs at $d_1$.
  In this context condition \eqref{eq:unitary-thompson-condition} can be rewritten as
  \begin{equation*}
    (1-d_1) \le \sum_{j=2}^{\infty} (1-d_j) < \infty.
  \end{equation*}
  Moreover, we can even assume the sequence $\mathbf{d}$ is nondecreasing.

  Let $N$ be the smallest positive integer $k$ (necessarily greater than one) which satisfies
  \begin{equation}
    \label{eq:smallest-N}
    \sum_{j=1}^k (1-d_j) > \sum_{j=k+1}^{\infty} (1-d_j).
  \end{equation}

  \noindent \emph{Claim:}
  There exists a finite sequence $\mathbf{s} = (s_j)_{j=1}^N \ge 0$ for which
  \begin{enumerate}[label=(\alph*)]
  \item
    \label{item:uts-tilde-d-increasing}
    $\mathbf{s}$ is nondecreasing and bounded above by one;
  \item
    \label{item:uts-bar-d-thompson-tilde-d}
    $\mathbf{\bar{d}}^{*} \prec_T \mathbf{s}^{*}$, where $\mathbf{\bar d} := (d_j)_{j=1}^N$;
  \item
    \label{item:uts-tilde-d-majorizes-tail-d}
    $\mathbf{\hat{d}} \prec ((\mathbf{1-s}) \oplus \mathbf{0})$, where $\mathbf{\hat{d}} := (1-d_{N+j})_{j=1}^{\infty}$.
  \end{enumerate}
  Note that because $\mathbf{\bar{d}},\mathbf{s}$ are nondecreasing finite sequences, their nonincreasing rearrangements $\mathbf{\bar{d}}^{*},\mathbf{s}^{*}$ simply reverse the order.

  \emph{Proof of claim:}
  Let $M$ denote the smallest positive integer $k$ satisfying
  \begin{equation}
    \label{eq:smallest-M}
    \sum_{j=1}^k (1-d_j) > \sum_{j=N+1}^{\infty} (1-d_j).
  \end{equation}
  The set of $k \in \mathbb{N}$ satisfying \eqref{eq:smallest-M} is nonempty because it contains $N$, and therefore we also have $M \le N$.
  Then for $1 \le k \le N$ define
  \begin{equation*}
    s_k :=
    \begin{cases}
      d_k                                                          & \text{if $k<M$} \\
      1 + \sum_{j=1}^{M-1} (1-d_j) - \sum_{j=N+1}^{\infty} (1-d_j) & \text{if $j=M$} \\
      1                                                            & \text{if $k>M$.}
    \end{cases}
  \end{equation*}
  For $M=1$ regard $\sum_{j=1}^{M-1} (1-d_j)$ as an empty sum.
  Note that $s_M \le 1$ by our choice of $M$.
  Moreover,
  \begin{equation*}
    s_M - d_M = \sum_{j=1}^M (1-d_j) - \sum_{j=N+1}^{\infty} (1-d_j) > 0,
  \end{equation*}
  and hence $\mathbf{s}$ is nondecreasing and bounded above by one because the same is true of $\mathbf{d}$, thereby establishing condition \ref{item:uts-tilde-d-increasing}.

  In fact, this additionally shows $\mathbf{\bar{d}} \le \mathbf{s}$ and hence $\mathbf{\bar{d}}^{*} \prec_w \mathbf{s}^{*}$.
  If $M > 1$, then $s_1 = d_1$ and so $\mathbf{\bar{d}}^{*} \prec_T \mathbf{s}^{*}$ trivially.
  If $M = 1$, then
  \begin{equation*}
    s_1 - d_1 = (1-d_1) - \sum_{j=N+1}^{\infty} (1-d_j)
    \le \sum_{j=2}^{\infty} (1-d_j) - \sum_{j=N+1}^{\infty} (1-d_j)
    = \sum_{j=2}^N (1-d_j)
    = \sum_{j=2}^N (s_j-d_j).
  \end{equation*}
  Therefore $\mathbf{\bar{d}}^{*} \prec_T \mathbf{s}^{*}$, thereby proving condition \ref{item:uts-bar-d-thompson-tilde-d}.

  For condition \ref{item:uts-tilde-d-majorizes-tail-d}, if $1 \le k < M$ then
  \begin{equation*}
    \sum_{j=N+1}^{N+k} (1-d_j)
    \le \sum_{j=1}^k (1-d_j)
    = \sum_{j=1}^k (1-s_j),
  \end{equation*}
  and for $k \ge M$,
  \begin{equation*}
    \sum_{j=N+1}^{N+k} (1-d_j)
    \le \sum_{j=N+1}^{\infty} (1-d_j)
    = \sum_{j=1}^M (1-s_j)
  \end{equation*}
  by the definitionn of $\mathbf{s}$, particularly $s_M$.
  Furthermore, since $s_k = 1$ if $M < k \le N$ then we can replace $M$ with $N$ in the last term of the above display which establishes condition \ref{item:uts-tilde-d-majorizes-tail-d}.

  Finally we construct the promised unitary operator with real-valued entries.
  Fix any basis $\mathfrak{e} = \{ e_j \}_{j=1}^{\infty}$ and let $U = \diag (\e^{\im\theta_1},\e^{-\im\theta_1}, \ldots, \e^{\im\theta_N}, \e^{-\im\theta_N}, 1, 1, 1, \ldots)$ where $\theta_j := \arccos s^{*}_j$ for $1 \le j \le N$.
  Then for $1 \le j \le 2N$, define
  \begin{equation*}
    f_j :=
    \begin{cases}
      \frac{e_j + e_{j+1}}{\sqrt{2}}      & \text{if $j$ is odd}   \\
      \frac{\im(e_{j-1} - e_j)}{\sqrt{2}} & \text{if $j$ is even.} \\
    \end{cases}
  \end{equation*}
  Then $U$ restricted to $\spans \{f_{2j-1},f_{2j}\}$ is
  \begin{equation*}
    \begin{pmatrix}
      \cos \theta_j  & \sin \theta_j \\
      -\sin \theta_j & \cos \theta_j \\
    \end{pmatrix}
    =
    \begin{pmatrix}
      s^{*}_j        & \sin \theta_j \\
      -\sin \theta_j & s^{*}_j       \\
    \end{pmatrix}
  \end{equation*}
  relative to $\{f_{2j-1},f_{2j}\}$.
  Then with respect to the basis $\mathfrak{f} := \{ f_{2j-1} \}_{j=1}^N \cup \{ f_{2j} \}_{j=1}^N \cup \{ e_j \}_{j=2N+1}^{\infty}$, $U$ has the form
  \begin{equation*}
    U =
    \begin{pmatrix}
      \diag \mathbf{s}^{*}          & \diag(\sin \theta_j)_{j=1}^N & 0 \\
      \diag(-\sin \theta_j)_{j=1}^N & \diag \mathbf{s}^{*}         & 0 \\
      0                             & 0                            & I \\
    \end{pmatrix},
  \end{equation*}
  which is orthogonal.

  Because $\mathbf{\bar{d}}^{*} \prec_T \mathbf{s}^{*}$ (due to \ref{item:uts-bar-d-thompson-tilde-d}), we can use Thompson's theorem to obtain orthogonal matrices $V,W$ acting on $M_N(\mathbb{C})$ so that $V(\diag \mathbf{s}^{*})W$ has diagonal $\mathbf{\bar{d}}^{*}$.
  Therefore with respect to the basis $\mathfrak{f}$, the orthogonal matrix $\tilde{U} := (V \oplus I \oplus I) U (W \oplus I \oplus I)$ has the form
  \begin{equation*}
    \tilde{U} =
    \begin{pmatrix}
      \diag \mathbf{\bar{d}}^{*} & *                    & 0 \\
      *                          & \diag \mathbf{s}^{*} & 0 \\
      0                          & 0                    & I \\
    \end{pmatrix}.
  \end{equation*}

  Finally, we consider the compression of $\tilde{U}$ to the subspace $\Kil := \spans \{ f_{2j-1} \}_{j=1}^N$ and its complement.
  We note that the compression $U_1$ of $\tilde{U}$ to $\Kil$ has diagonal $\mathbf{\bar{d}}^{*}$ and the compression $U_2$ of $\tilde{U}$ to $\Kil^{\perp}$ is $\diag \mathbf{s}^{*} \oplus I$.
  The operator $I_{\Kil^{\perp}} - U_2$ is thus a finite rank positive operator with singular value sequence $((\mathbf{1-s}) \oplus \mathbf{0})$.
  Because $\mathbf{\hat{d}} \prec ((\mathbf{1-s}) \oplus \mathbf{0})$ (due to condition \ref{item:uts-tilde-d-majorizes-tail-d}) we can apply the Schur--Horn theorem  (\autoref{thm:compact-schur-horn}, which can be achieved via an orthogonal unitary matrix) to conclude that $\mathbf{\hat{d}}$ is a diagonal of $I_{\Kil^{\perp}} - U_2$.
  Therefore $(d_j)_{j=N+1}^{\infty}$ is a diagonal of $U_2$.
  By \autoref{lem:compression-diagonal} we find that $\mathbf{d} = \mathbf{\bar{d}}^{*} \oplus (d_j)_{j=N+1}^{\infty}$ is a diagonal of $\tilde{U}$ and is achieved over real Hilbert space.
\end{proof}

\section{Extremal cases and selfadjoint operators}

In the finite dimensional setting, Thompson's theorem (\autoref{thm:thompson}) has another surprise in store when the final inequality is tight (i.e., the two sides are actually equal).
Certainly, when the inequality is tight, $\mathbf{\tilde{d}} := \big( \abs{d_1},\ldots,\abs{d_{N-1}},-\abs{d_N} \big) \prec (s_1,\ldots,s_{N-1},-s_N)$.
Then by the Schur--Horn theorem (\autoref{thm:schur-horn}), there is a selfadjoint matrix $A$ with diagonal $\mathbf{\tilde{d}}$ and singular value sequence $\mathbf{s}$.
However, Thompson's work \cite[Proof of Lemma 5]{Tho-1977-SJAM} guarantees a sort of converse: any matrix $A$ with diagonal $\mathbf{\tilde{d}}$ so that $\abs{\mathbf{\tilde{d}}} > 0$ and singular value sequence $\mathbf{s}$ is selfadjoint.
The next lemma is one of the key tools in Thompson's proof.

\begin{lemma}[Thompson \protect{\cite[Lemma 3]{Tho-1977-SJAM}}]
  \label{lem:2x2-trace-equal-trace-norm}
  If $A \in M_2(\mathbb{C})$ has nonnegative diagonal entries $d_1, d_2$ and singular values $s_1, s_2$, then $s_1 + s_2 \ge d_1 + d_2$ with equality if and only if $A$ is positive.
\end{lemma}

Note that this is one example where we can conclude an operator is selfadjoint based on its diagonal.
The generalization of this to trace-class operators is a simple consequence of the Cauchy--Schwarz inequality.
In fact, the proof given below works in any semifinite von Neumann algebra with a faithful normal semifinite trace.
This proof appeared for type II$_1$ factors in the work of Kennedy and Skoufranis \cite{KS-2014}, who attribute their proof to David Sherman.
Here we present the $B(\Hil)$ version.

\begin{theorem}
  \label{thm:finite-trace-norm-inequality}
  If $A \in B(\Hil)$ is trace-class, then $\abs{\trace A} \le \trace \abs{A}$ with equality if and only if $cA$ is positive for some scalar $\abs{c} =1$.
\end{theorem}

\begin{proof}
  Let $A = U\abs{A}$ be the polar decomposition, so that $U^{*}U$ is the projection onto the range of $\abs{A}$.
  Note that $\abs{A}^{\frac{1}{2}}$ is Hilbert--Schmidt since $A$ (equivalently, $\abs{A}$) is trace-class.
  Moreover, given $A,B$ Hilbert--Schmidt, the mapping $(A,B) \mapsto \trace(B^{*}A)$ is an inner product.
  Therefore by the Cauchy--Schwarz inequality
  \begin{equation*}
    \abs{\trace A} = \abs{\trace ( U\abs{A}^{\frac{1}{2}} \abs{A}^{\frac{1}{2}} ) } \le \trace ( \abs{A}^{\frac{1}{2}} U^{*}U \abs{A}^{\frac{1}{2}} )^{\frac{1}{2}} ( \trace \abs{A} )^{\frac{1}{2}} = \trace \abs{A},
  \end{equation*}
  with equality if and only if $U\abs{A}^{\frac{1}{2}} = c\abs{A}^{\frac{1}{2}}$ for some scalar $c$, and hence $A = c\abs{A}$.
  Since $\abs{A}^2 = A^{*}A = \abs{c}^2 \abs{A}^2$, $\abs{c} = 1$ as long as $A \not= 0$.
  Of course, the result is trivially true for $A = 0$.
\end{proof}

We also require a basic fact.

\begin{lemma}
  \label{lem:liminf-sum-decreasing}
  For real-valued sequences $(a_n), (b_n)$ with $(b_n)$ nonincreasing,
  \begin{equation*}
    \liminf (a_n + b_n) = \liminf a_n + \inf b_n.
  \end{equation*}
\end{lemma}

\begin{proof}
  Trivially, $\liminf (a_n + b_n) \le \liminf a_n + \liminf b_n$, but also
  \begin{equation*}
    \liminf a_n + \liminf b_n = \liminf a_n + \inf b_n = \liminf (a_n + \inf b_n) \le \liminf (a_n + b_n). \qedhere
  \end{equation*}
\end{proof}

\begin{lemma}
  \label{lem:infinite-triangle-inequality}
  If $\mathbf{d}$ is a sequence of complex numbers for which
  \begin{equation*}
    \liminf_{n \to \infty} \vpren{ \sum_{i=1}^n \abs{d_i} - \vabs{ \sum_{i=1}^n d_i } } = 0,
  \end{equation*}
  then $\mathbf{d}$ has constant phase (i.e., $\frac{d_j}{\abs{d_j}} = \frac{d_k}{\abs{d_k}}$ whenever $d_j \not= 0 \not= d_k$).
\end{lemma}

\begin{proof}
  Consider any $j,k \in \mathbb{N}$ for which $d_j \not= 0 \not= d_k$.
  Then for $n \ge \max \{j,k\}$ we have
  \begin{equation*}
    \begin{array}{*5{>{\displaystyle}c}}
      \sum_{i=1}^n \abs{d_i} - \vabs{ \sum_{i=1}^n d_i }
      & \ge & \vpren{ \abs{d_j} + \abs{d_k} +  \sum_{\substack{i=1 \\ i\not=j,k}}^n \abs{d_i} }  & - & \vpren{ \abs{d_j + d_k} + \vabs{ \sum_{\substack{i=1 \\ i\not=j,k}}^n d_i } } \\[2.5em]
      & \ge & \abs{d_j} + \abs{d_k} & - & \abs{d_j + d_k}. \\
    \end{array}
  \end{equation*}
  Taking the limit inferior as $n \to \infty$ yields $\abs{d_j} + \abs{d_k} = \abs{d_j + d_k}$, which holds if and only if $d_j$ and $d_k$ have the same phase.
\end{proof}

We now prove an extension to compact operators of Thompson's result that an operator with diagonal $\tilde{d}$ is selfadjoint if the final inequality in \autoref{thm:thompson} is tight.
Additionally, it is a generalization of the finite dimensional result \autoref{thm:finite-trace-norm-inequality} above.
The limit inferior condition which appears below says precisely that $\abs{\mathbf{d}}^{*} \preccurlyeq \mathbf{s}$, i.e., $\abs{\mathbf{d}}^{*}$ is \emph{strongly majorized} by $\mathbf{s}$ in the sense of \cite[Definition 1.2]{KW-2010-JFA}.

\begin{theorem}
  \label{thm:compact-thompson-tight-case}
  If $\mathbf{d}$ is a diagonal of a compact operator $A$ with singular value sequence $\mathbf{s}$ and
  \begin{equation}
    \label{eq:strong-majorization}
    \liminf_{n \to \infty} \sum_{j=1}^n (s_j - |d|^{*}_j) = 0,
  \end{equation}
  then $A = UB$ for some diagonal unitary operator $U$ and positive compact operator $B$.
\end{theorem}

\begin{proof}
  By \autoref{prop:nonnegative-diagonals-suffice} it suffices to prove that $A$ is positive where $\mathbf{d} = \mathbf{\abs{d}} \in c_0^+$.
  Let $\{e_k\}_{k=1}^{\infty}$ be an orthonormal basis in which $A$ has diagonal $\mathbf{d}$.

  The diagonal $\mathbf{d}$ is maximal in the sense that if $\mathbf{\tilde{d}} \in c_0^+$ is another diagonal of a compact operator with singular value sequence $\mathbf{s}$ which differs from $\mathbf{d}$ on a finite index set $F$ and the zero set of $\mathbf{d}$ contains the zero set of $\mathbf{\tilde{d}}$, then $\sum_{i \in F} \tilde{d}_i \le \sum_{i \in F} d_i$.
  To prove this claim let $\phi, \tilde{\phi} : \mathbb{N} \to \mathbb{N}$ be the injections which produce the nonincreasing rearrangements $\mathbf{d}^{*}, \mathbf{\tilde{d}}^{*}$.
  That is, $d^{*}_i = d_{\phi(i)}$ and similarly for $\tilde{\phi}, \tilde{d}$ (recall \autoref{not:background-notation} that nonincreasing rearrangements of $c_0^+$ sequences of infinite support eliminate the zeros).
  Then for $n \ge \max \phi^{-1}(F) \cup \tilde{\phi}^{-1}(F)$ and $r := \abs{F \setminus \phi(\mathbb{N})}$, since $i \in F\setminus \phi(\mathbb{N})$ implies $d_i = 0$, we have
  \begin{equation*}
    \sum_{i=1}^n \tilde{d}^{*}_i - \sum_{i \in F} \tilde{d}_i = \sum_{i=1}^{n-r} d^{*}_i - \sum_{i \in F} d_i.
  \end{equation*}
  Using this equation and \autoref{prop:fans-theorem}, we conclude
  \begin{equation*}
    0 \le \sum_{i=1}^n ( s_i - \tilde{d}^{*}_i )  = \sum_{i=1}^n ( s_i - d^{*}_i ) + \sum_{i=1}^r  d^{*}_{n-r+i} + \sum_{i \in F} ( d_i - \tilde{d}_i ).
  \end{equation*}
  Rearranging this inequality yields
  \begin{equation*}
    \sum_{i \in F} \tilde{d}_i \le \sum_{i=1}^n ( s_i - d^{*}_i ) + \sum_{i=1}^r d^{*}_{n-r+i} + \sum_{i \in F} d_i.
  \end{equation*}
  and taking the limit inferior we obtain
  \begin{align*}
    \sum_{i \in F} \tilde{d}_i
    &\le \liminf_{n \to \infty} \vpren{ \sum_{i=1}^n ( s_i - d^{*}_i ) + \sum_{i=1}^r d^{*}_{n-r+i} } + \sum_{i \in F} d_i  \\
    &= \liminf_{n \to \infty} \vpren{ \sum_{i=1}^n ( s_i - d^{*}_i ) } + \lim_{n \to \infty} \vpren{ \sum_{i=1}^r d^{*}_{n-r+i} } + \sum_{i \in F} d_i \\
    &= \sum_{i \in F} d_i.
  \end{align*}

  Then for $i\not=j$ consider the compression
  \(
  A_{i,j} :=
  \begin{psmallmatrix}
    d_i & b \\
    c & d_j \\
  \end{psmallmatrix}
  \)
  of $A$ to $\spans \{e_i,e_j\}$.
  Let $\sigma_1,\sigma_2$ be the singular values of $A_{i,j}$.
  Using the singular value decomposition we can find unitaries $U,V \in M_2(\mathbb{C})$ so that $UA_{i,j}V = \diag(\sigma_1,\sigma_2)$.
  Then for $\tilde{U} := U \oplus I$ and $\tilde{V} := V \oplus I$, the operator $\tilde{A} := \tilde{U}A\tilde{V}$ has diagonal $\mathbf{\tilde{d}}$ which is precisely $\mathbf{d}$ except $d_i,d_j$ are replaced by $\sigma_1,\sigma_2$.
  Because $\mathbf{d}$ is maximal, $\sigma_1 + \sigma_2 \le d_i + d_j$ and so we have equality which by \autoref{lem:2x2-trace-equal-trace-norm} implies $A_{i,j}$ is selfadjoint.
  Since $e_i,e_j$ were arbitrary, this means that $A$ is selfadjoint.

  To prove $A$ is positive, we first split it into its positive and negative parts $A = A_+ - A_-$.
  Then $\mathbf{s^+} := s(A_+) \le s(A_+ \oplus A_-) = s(A) = \mathbf{s}$.
  Let $\mathbf{d^+}$ be the diagonal of $A_+ = A + A_-$ in the basis $\{e_j\}_{j=1}^{\infty}$, and notice that $\mathbf{d^+} \ge \mathbf{d}$ because $A_- \ge 0$.
  By \autoref{prop:fans-theorem} we have
  \begin{equation*}
    0 \le \sum_{j=1}^n \big( s^+_j - (d^+)^{*}_j \big) \le \sum_{j=1}^n ( s^+_j - d^+_j ) = \sum_{j=1}^n ( s^+_j - s_j ) + \sum_{j=1}^n ( s_j - d_j ) + \sum_{j=1}^n ( d_j - d^+_j ).
  \end{equation*}
  Rearranging we obtain
  \begin{equation*}
    \sum_{j=1}^n ( d^+_j - d_j ) \le \sum_{j=1}^n ( s^+_j - s_j ) + \sum_{j=1}^n ( s_j - d_j ).
  \end{equation*}
  The sequence $\sum_{j=1}^n ( d^+_j - d_j )$ is nondecreasing and $\sum_{j=1}^n ( s^+_j - s_j )$ is nonincreasing, so taking the limit inferior of the preceding inequality and applying \eqref{eq:strong-majorization} and  \autoref{lem:liminf-sum-decreasing} yields
  \begin{equation*}
    0 \le \sup_n \sum_{j=1}^n ( d^+_j - d_j ) \le \inf_n \sum_{j=1}^n ( s^+_j - s_j ) \le 0.
  \end{equation*}
  Therefore $\mathbf{d^+} = \mathbf{d}$ and so the diagonal of $A_-$ is constantly zero, which implies $A_- = 0$ since $A_-$ is positive (i.e., because the conditional expectation onto the diagonal masa is faithful).
\end{proof}

\begin{remark}
  Recall that \autoref{thm:compact-schur-horn} guarantees that a nonincreasing sequence $\mathbf{d}>0$ is the diagonal of a positive compact operator with singular value sequence (in this case, eigenvalue sequence) $\mathbf{s}$ if $\mathbf{d} \prec \mathbf{s}$, which means $\mathbf{d} \prec_w \mathbf{s}$ and $\sum_{j=1}^{\infty} d_j = \sum_{j=1}^{\infty} s_j$.
  Since $\mathbf{d} \prec_w \mathbf{s}$ is guaranteed for any diagonal $\mathbf{d}$ by \autoref{prop:fans-theorem} (even for nonpositive operators), we should compare equality of these sums to condition \eqref{eq:strong-majorization}.
  It is clear that \eqref{eq:strong-majorization} implies $\sum_{j=1}^{\infty} d_j = \sum_{j=1}^{\infty} s_j$, but they are certainly not equivalent (consider $\mathbf{s} := (1/j)_{j=1}^{\infty}$ and $\mathbf{d} := (c/j)_{j=1}^{\infty}$ for $0<c<1$).
  Moreover, in \autoref{thm:compact-thompson-tight-case} the requirement that the limit inferior is zero cannot be weakened.
  Examination of the proofs of Cases 2 and 3 in \autoref{thm:compact-thompson-sufficiency} shows that when the limit inferior is nonzero we can sometimes choose the operator to be nonselfadjoint.
\end{remark}

Finally, we obtain an infinite dimensional analogue of \autoref{thm:finite-trace-norm-inequality} for both trace-class and non-trace-class operators.
For a complex-valued sequence $\mathbf{d} \in c_0$, we would like to rearrange $\mathbf{d}$ in order of nonincreasing modulus.
Of course, there are two problems with this.
Firstly, it is not possible to place $\mathbf{d}$ in order of nonincreasing modulus if it has infinite support and some zero terms.
We will deal with this case in the same way as defining $\abs{d}^{*}$ from $\abs{d}$; we ignore the zeros if it has infinite support.
Secondly, such a rearrangement is nonunique if there exist two unequal entries in the sequence with the same modulus.
Fortunately, nonuniqueness is not an issue for us because any such sequence will suffice for our purposes.
Let $\mathbf{d}^{\dag}$ denote any sequence satisfying $d^{\dag}_j = d_{\phi(j)}$ where $\phi$ is an injective function for which $\abs{d}^{*}_j = \abs{d_{\phi(j)}}$.
In other words, $\phi$ implements a nonincreasing rearrangement of $\abs{\mathbf{d}}$.
Note that $\abs{\mathbf{d}^{\dag}} = \abs{\mathbf{d}}^{*}$.

The next corollary is the analogue of \autoref{thm:finite-trace-norm-inequality} for compact operators.

\begin{theorem}
  \label{thm:compact-trace-norm-inequality-tight}
  If $A$ is a compact operator with diagonal $\mathbf{d}$, singular value sequence $\mathbf{s}$, and
  \begin{equation*}
    \liminf_{n \to \infty} \vpren{ \sum_{i=1}^n s_i - \vabs{ \sum_{i=1}^n d^{\dag}_i } } = 0
  \end{equation*}
  for some choice of rearrangement $\mathbf{d}^{\dag}$ of $\mathbf{d}$ in order of nonincreasing modulus, then $cA$ is positive for some scalar with $\abs{c} = 1$.
  In particular, if $A$ is trace-class, then $\trace \abs{A} = \abs{\trace A}$ implies $cA \ge 0$ for some $\abs{c} = 1$.
\end{theorem}

\begin{proof}
  By the triangle inequality and \autoref{prop:fans-theorem}, we have
  \begin{equation*}
    0 \le \liminf_{n \to \infty} \vpren{ \sum_{i=1}^n s_i - \sum_{i=1}^n \vabs{d^{\dag}_i} } \le \liminf_{n \to \infty} \vpren{ \sum_{i=1}^n s_i - \vabs{ \sum_{i=1}^n d^{\dag}_i } } \le 0.
  \end{equation*}
  Along with the hypothesis and basic properties of the limits inferior and superior, this yields
  \begin{align*}
    0
    &\le \liminf_{n \to \infty} \vpren{ \sum_{i=1}^n \vabs{d^{\dag}_i} - \vabs{ \sum_{i=1}^n d^{\dag}_i }} \displaybreak[0]\\
    &= \liminf_{n \to \infty} \vpren{ \vpren{ \sum_{i=1}^n s_i - \vabs{ \sum_{i=1}^n d^{\dag}_i } } - \vpren{ \sum_{i=1}^n s_i - \sum_{i=1}^n \vabs{d^{\dag}_i} } } \\
    &\le \liminf_{n \to \infty} \vpren{ \sum_{i=1}^n s_i - \vabs{ \sum_{i=1}^n d^{\dag}_i } } + \limsup_{n \to \infty} -\vpren{ \sum_{i=1}^n s_i - \sum_{i=1}^n \vabs{d^{\dag}_i} }  \\
    &= \liminf_{n \to \infty} \vpren{ \sum_{i=1}^n s_i - \vabs{ \sum_{i=1}^n d^{\dag}_i } } - \liminf_{n \to \infty} \vpren{ \sum_{i=1}^n s_i - \sum_{i=1}^n \vabs{d^{\dag}_i} }  \\
    &= 0.
  \end{align*}
  So by \autoref{lem:infinite-triangle-inequality}, there is some $\abs{c} = 1$ for which $c\mathbf{d} = \abs{\mathbf{d}}$.
  Note $\abs{\mathbf{d}}$ is the diagonal of the compact operator $cA$ and $s(cA) = \mathbf{s}$ and satisfies the hypotheses of \autoref{thm:compact-thompson-tight-case} since $\abs{\mathbf{d}}^{*} = \abs{\mathbf{d}^{\dag}}$.
  Thus $cA$ is positive.

  When $A$ is trace-class, \autoref{thm:compact-trace-norm-inequality-tight} provides a verbatim generalization of \autoref{thm:finite-trace-norm-inequality}.
  Indeed, this is because if $A$ is trace-class then
  \begin{equation*}
    \trace \abs{A} - \abs{ \trace A } = \lim_{n \to \infty} \vpren{ \sum_{i=1}^n s_i - \vabs{ \sum_{i=1}^n d^{\dag}_i } }. \qedhere
  \end{equation*}
\end{proof}

In order to address the question of the diagonals of a possibly selfadjoint unitary operator in the extremal (equality) case of \autoref{thm:unitary-thompson-sufficiency} (condition \eqref{eq:unitary-thompson-condition}), we need another $2 \times 2$ lemma due to Thompson.
Departing slightly from the case $\mathbf{d} \ge 0$ to a single negative diagonal entry, equality in \autoref{thm:unitary-tight-selfadjoint} forces selfadjointness.

\begin{lemma}[Thompson \protect{\cite[Lemma 4]{Tho-1977-SJAM}}]
  \label{lem:2x2-tight-thompson-selfadjoint}
  If $A \in M_2(\mathbb{C})$ has diagonal entries $d_1 > 0 > d_2$ with $d_1 \ge \abs{d_2}$ and singular values $s_1 \ge s_2$, then $s_1 - s_2 \ge d_1 - \abs{d_2}$ with equality if and only if $A$ is selfadjoint.
\end{lemma}

\begin{theorem}
  \label{thm:unitary-tight-selfadjoint}
  If $\mathbf{d} = (-d_1,d_2,d_3,\ldots)$ where $d_j > 0$ for all $j \ge 1$ satisfying
  \begin{equation*}
    2(1-d_1) = \sum_{i=1}^{\infty} (1-d_i)
  \end{equation*}
  is the diagonal of a unitary and any such unitary is necessarily selfadjoint.
\end{theorem}

\begin{proof}
  The fact that $\mathbf{d}$ is the diagonal of a unitary $U$ is a consequence of \autoref{thm:unitary-thompson-sufficiency}.
  The rest of the proof is analogous to the proof of \autoref{thm:compact-thompson-tight-case}, but we need both \autoref{lem:2x2-trace-equal-trace-norm} and \autoref{lem:2x2-tight-thompson-selfadjoint}.
  The sequence $\mathbf{d}$ is extremal in the following sense.
  If $\mathbf{\tilde{d}} = (-\tilde{d}_1,\tilde{d}_2,\tilde{d}_3,\ldots)$ where $\tilde{d}_j > 0$ is another real-valued sequence of a unitary operator with a single negative entry in the first coordinate and disagrees with $\mathbf{d}$ on a finite index set $F$, then
  \begin{equation}
    \label{eq:extreme-unitary-diagonal}
    2(1-d_1) - \sum_{i \in F} (1-d_i) \ge 2(1-\tilde{d}_1) - \sum_{i \in F} (1-\tilde{d}_i).
  \end{equation}
  Indeed,
  \begin{align*}
    0
    &= 2(1-d_1) - \sum_{i=1}^{\infty} (1-d_i) \\
    &= 2(1-d_1) - \sum_{i \in F} (1-d_i) - \sum_{i \notin F} (1-d_i) \\
    &= 2(1-d_1) - \sum_{i \in F} (1-d_i) - \sum_{i \notin F} (1-\tilde{d}_i) \\
    &= 2(1-d_1) - \sum_{i \in F} \Big( (1-d_i) - (1-\tilde{d}_i) \Big) - \sum_{i=1}^{\infty} (1-\tilde{d}_i)  \\
    &= 2\big((1-d_1)-(1-\tilde{d}_1)\big) - \sum_{i \in F} \Big( (1-d_i) - (1-\tilde{d}_i) \Big) + 2(1-\tilde{d}_1) - \sum_{i=1}^{\infty} (1-\tilde{d}_i)  \\
    &\le 2\big((1-d_1)-(1-\tilde{d}_1)\big) - \sum_{i \in F} \Big( (1-d_i) - (1-\tilde{d}_i) \Big),
  \end{align*}
  where the first equality is by hypothesis and the inequality is due to \autoref{thm:unitary-thompson-sufficiency}.

  Now, take any pair of diagonal entries $d_j, d_k$ with $j,k > 1$, and look at the compression $U_{j,k}$ of $U$ to $\spans\{e_j,e_k\}$.
  $U_{j,k}$ has singular values $\sigma_1, \sigma_2$, and using the singular value decomposition of $U_{j,k}$ we can multiply $U$ on the left and right by unitaries to get a new unitary whose diagonal is $\mathbf{d}$ with $d_j, d_k$ replaced by $\sigma_1,\sigma_2$.
  By \eqref{eq:extreme-unitary-diagonal}, setting $\tilde{d}_j = \sigma_1, \tilde{d}_k = \sigma_2$ we conclude that $\sigma_1 + \sigma_2 \le d_j + d_k$, and therefore by \autoref{lem:2x2-trace-equal-trace-norm}, $U_{j,k}$ is selfadjoint.

  Next consider $U_{1,j}$ with singular values $\sigma_1 \ge \sigma_2$.
  Note that \eqref{eq:extreme-unitary-diagonal} guarantees $d_1 \le d_j$ for all $j \ge 2$.
  Then using the singular value decomposition of $U_{1,j}$ and multiplying also by
  \(
  \begin{psmallmatrix}
    -1            & 0 \\
    \hphantom{-}0 & 1 \\
  \end{psmallmatrix}
  \),
  we can multiply $U$ on the left and right by unitaries to get a new unitary  whose diagonal is $\mathbf{d}$ with $d_1, d_j$ replaced by $\sigma_1 - \sigma_2$.
  By \eqref{eq:extreme-unitary-diagonal}, we have that $\sigma_1 - \sigma_2 \le d_j - d_1$, and so by \autoref{lem:2x2-tight-thompson-selfadjoint} $U_{1,j}$ is selfadjoint.
  Since $U_{j,k}$ is selfadjoint for all $j,k \in \mathbb{N}$, $U$ is selfadjoint.
\end{proof}

\section{Open questions}
\label{sec:open-questions}

\begin{question}
  Characterize the diagonals of normal operators with specified singular values.
\end{question}

\begin{question}
  Characterize the diagonals of partial isometries with kernel dimension $n$ and co-kernel dimension $m$ ($1 \le n,m \le \infty$).
  In particular, find the diagonals of what we call \emph{square} partial isometries.
  An operator $A$ is said to be \emph{square} if $A^{*}A$ and $AA^{*}$ have the same dimension of their kernels.
\end{question}

\begin{question}
  Characterize the diagonals of positive compact operators with one dimensional kernel, and then those with finite dimensional kernel.
\end{question}

\emergencystretch=4em
\printbibliography

\end{document}